\documentclass[12pt]{amsart}
\usepackage[sans]{dsfont}
\usepackage{amsfonts,amssymb,amsbsy,amsmath,amsthm,enumerate}

\numberwithin{equation}{section}

\newtheorem{theorem}{Theorem}[section]
\newtheorem{proposition}[theorem]{Proposition}
\newtheorem{lemma}[theorem]{Lemma}

\theoremstyle{definition}

\newcommand{\bel}{\begin{equation} \label}
\newcommand{\ee}{\end{equation}}

\newcommand{\re}{{\mathbb R}}

\newcommand{\N}{{\mathbb N}}

\newcommand{\eps}{{\varepsilon}}
\newcommand{\obr}{{\mathcal O}_R}

\def\beq{\begin{equation}}
\def\eeq{\end{equation}}
\newcommand{\bea}{\begin{eqnarray}}
\newcommand{\eea}{\end{eqnarray}}
\newcommand{\beas}{\begin{eqnarray*}}
\newcommand{\eeas}{\end{eqnarray*}}

\begin{document}
\title[Schr\"odinger operators with oscillating decaying potentials]{Discrete spectrum for Schr\"odinger operators with oscillating decaying potentials}

\author[G.~Raikov]{Georgi Raikov}

\begin{abstract}
We consider the Schr\"odinger operator $H_{\eta W} = -\Delta + \eta W$, self-adjoint in $L^2(\re^d)$, $d \geq 1$. Here $\eta$ is a non constant oscillating function, while $W$ decays slowly and regularly at infinity. We study the asymptotic behaviour of the discrete spectrum of $H_{\eta W}$ near the origin, and due to the irregular decay of $\eta W$, we encounter some non semiclassical phenomena. In particular, $H_{\eta W}$ has less eigenvalues than suggested by the semiclassical intuition.

\end{abstract}

\maketitle

{\bf  AMS 2010 Mathematics Subject Classification:} 35P20;   81Q10,
35J10\\

{\bf  Keywords:}
 Schr\"odinger operators, oscillating decaying  potentials, discrete spectrum\\

\section{Introduction}
\label{s1} \setcounter{equation}{0}
In this note, we study the asymptotic behaviour of the negative eigenvalues of the Schr\"odinger operator
    \bel{may60}
    H_V : = -\Delta + V,
    \ee
    self-adjoint in $L^2(\re^d)$, $d \geq 1$. Here, $V : \re^d \to \re$ is a suitable bounded  potential such that $\lim_{|x| \to \infty} V(x) = 0$. We assume that
    \bel{j43}
    V = \eta W,
    \ee
    where the factor $W$ decays regularly at infinity, i.e. its derivatives decay faster than $W$ itself, while the factor $\eta$ is, for example, an appropriate non constant almost periodic function. Thus, the  product $\eta W$ in \eqref{j43} does not decay regularly at infinity.\\
    One-dimensional Schr\"odinger operators defined by \eqref{may60} - \eqref{j43} with periodic factor $\eta$ arose as effective Hamiltonians in \cite{R1}, where the asymptotic distribution of the discrete spectrum for waveguides with perturbed periodic twisting, was investigated. Moreover, multidimensional operators of related kind were discussed in \cite{den, LNS,  S, S3, S1, S2} where the problem about the location of the absolutely continuous spectrum $\sigma_{\rm ac}(H_V)$ of $H_V$ was attacked. A typical result in this direction is that if {\em both} operators $H_V$ and $H_{-V}$ do not have ``too many" negative eigenvalues, e.g. if
    \bel{may65}
    \int_{-\infty}^0 |E|^{-1/2} \left(N(E; V) + N(E; -V)\right) dE < \infty,
    \ee
    where $N(E;V)$ is the number of the eigenvalues of $H_V$, smaller than $E \leq 0$, and counted with the multiplicities, then $\sigma_{\rm ac}(H_V)$ is essentially supported on $[0,\infty)$ (see \cite{DR} for $d=1$ or \cite[Theorem 1.1]{S3} for $d \geq 1$). From this point of view, operators of form \eqref{may60} - \eqref{j43} with vanishing mean value of $\eta$ are of particular interest. \\

    Here we would like to also mention the articles \cite{skr, sas}, concerning {\em rapidly oscillating potentials} $V$; typical examples of such potentials are $V(x) = \cos{(|x|^2)}$ or $V(x) = (1+|x|^2)^{-1} e^{|x|} \sin{(e^{|x|})}$. Among the main problems attacked in \cite{skr} and \cite{sas}, are the self-adjointness and semi-boundedness of $H_V$, criteria for the validity of  $\sigma_{\rm ess}(H_V) = [0,\infty)$ and for the finiteness of the discrete spectrum of $H_V$, as well as various estimates of its negative eigenvalues. Although, formally, the class of the potentials $V$ considered in \cite{skr, sas} does not intersect with the corresponding class studied in the present article, there is a deep and not quite evident relation between the spectral properties of the two classes of Schr\"odinger operators $H_V$ (see e.g. the discussion in \cite[Section XI.8, Appendix 2]{RS3}).\\

    In the present article, we show that
    if the mean value $\eta_0$ of $\eta$ does not vanish, while $W$ decays regularly and admits the estimates\footnote{Here and in the sequel we write $A \asymp B$ if there exist constants $c_j \in (0,\infty)$, $j=1,2$, such that $c_1 A \leq B \leq c_2 A$.}
    \bel{may61}
    W(x) \asymp -|x|^{-\rho}, \quad |x| \geq R,
    \ee
    with $\rho \in (0,2]$, and $R \in (0,\infty)$,  then the effective Hamiltonian which governs the eigenvalue asymptotics for $H_{\eta W}$, is the Schr\"odinger operator $H_{\eta_0 W}$. More precisely, if $\eta_0 > 0$, and $\rho \in (0,2)$, then under appropriate assumptions we have
    $$
    N(E;\eta W) = N(E;\eta_0 W)\,(1 + o(1)) \asymp |E|^{d\left(\frac{1}{2} - \frac{1}{\rho}\right)}, \quad E \uparrow 0,
    $$
    while if $\eta_0 < 0$, then
    \bel{may62}
     N(E;\eta W) = O(1), \quad E \uparrow 0,
     \ee
     i.e. the operator $H_{\eta W}$ has not more than finitely many negative eigenvalues (see Theorem \ref{th4} (i)). Note that if $\eta_0 < 0$, then \eqref{may61} implies that the negative part of the potential $\eta W$ is compactly supported, and, hence,
     $$
     N(E;\eta_0 W) = O(1), \quad E \uparrow 0.
     $$
      Related results are obtained also in the border case $\rho = 2$ (see Theorem \ref{th4} (iii)). \\
      If $\eta_0 = 0$, and $W$ satisfies \eqref{may61} with $\rho \in (1,2]$, we show that \eqref{may62} holds true again, and if $\rho \in (0,1]$, we obtain asymptotic upper bounds of $N(E;\eta W)$ as $E \uparrow 0$ (see Theorem \ref{th4} (ii)). \\
     Under slightly more restrictive assumptions, we obtain the main asymptotic term of $N(E;\eta W)$ as $E \uparrow 0$ in the case $\eta_0 = 0$, and
     $\rho \in (0,1]$ (see Theorem \ref{th5} below). More precisely, we show that the effective Hamiltonian which governs the asymptotics of $N(E;\eta W)$ as $E \uparrow 0$ is the Schr\"odinger operator $H_{-\psi_0 W^2}$ with a suitable constant $\psi_0 > 0$. In particular, it follows from these results that the asymptotic bounds of Theorem \ref{th4} (ii), are sharp. \\

    The article is organized as follows. The next section contains auxiliary known results concerning the asymptotic distribution of the discrete spectrum of the operator $H_V$ with regularly decaying $V$. In Section \ref{s2} we formulate our main results, and briefly comment on it. The proofs of the main results can be found in Section \ref{s4}.

\section{Auxiliary Results}
\label{s1a} \setcounter{equation}{0}
This section contains well known estimates of the discrete spectrum of the Schr\"odinger operator $H_V$ with regularly decaying potential $V$, needed for the better understanding and for the proofs of our main results.
Assume that
    \bel{may66}
    V \in L^\infty(\re^d; \re), \quad \lim_{|x| \to \infty} V(x) = 0.
    \ee
 Then $V$ is relatively compact with respect to $H_0 = -\Delta$, and we have
$$
\sigma_{\rm ess}(H_V) = \sigma_{\rm ess}(H_0) = [0,\infty).
$$
Thus, the possible discrete spectrum of $H_V$ is negative, and it can accumulate only at the origin.

\begin{proposition} \label{mp1}
Let $d \geq 1$. Assume that \eqref{may66} holds true, and there exist constants $c \in (0,\infty)$ and $\rho \in (2,\infty)$ such that
\bel{m10}
V_-(x) \leq c(1+|x|)^{-\rho}, \quad x \in \re^d.
\ee
Then
\bel{m11}
N(0;V) < \infty.
\ee
\end{proposition}
If $d \geq 3$, the result follows from the Cwickel-Lieb-Rozenblum estimate
     (see e.g. \cite[Theorem XIII.12]{RS4}). If $d=2$, it is implied by \cite[Eq. (44)]{ckmw}, while for $d=1$ it follows from \cite[Problem 22, Chapter XIII]{RS4}.\\

{\em Remark}: For $E \geq 0$ set
    \bel{may67}
    N_{\rm cl}(E; V) : =
    (2\pi)^{-d} \left|\left\{(x,\xi) \in T^*\re^d \, | \, |\xi|^2 + V(x) < E\right\}\right| =
    \frac{\tau_d}{(2\pi)^d} \int_{\re^d} (V(x)-E)_-^{d/2} dx,
    \ee
    where $|\cdot|$ is the Lebesgue measure, and $\tau_d = \frac{\pi^{d/2}}{\Gamma(1+d/2)}$ is the volume of the unit ball in $\re^d$, $d \geq 1$. Note that if $V$ satisfies \eqref{m10} with $\rho > 2$, then $N_{\rm cl}(0; V) < \infty$.\\

The following proposition shows that the condition $\rho > 2$ in \eqref{m10} is close to the optimal one.
\begin{proposition} \label{mp2} {\rm (\cite[Theorem XIII.82]{RS4})}
Let $d \geq 1$, $\rho \in (0,2)$. Assume that there exist constants $C \in (0,\infty)$ and $R \in (0,\infty)$, such that
\bel{m12}
|V(x)| \leq C(1+|x|)^{-\rho}, \quad x \in \re^d,
\ee
\bel{m13}
|\nabla V(x)| \leq C(1+|x|)^{-\rho-1}, \quad x \in \re^d,
\ee
$$
V(x) \leq - C|x|^{-\rho}, \quad x \in \re^d, \quad |x| \geq R.
$$
Then we have
\bel{j15}
N(E;V) = N_{\rm cl}(E; V)
 \, (1 + o(1)) \asymp
|E|^{d(\frac{1}{2} - \frac{1}{\rho})}, \quad E \uparrow 0.
\ee
\end{proposition}

Next, we discuss the border-line case $\rho = 2$. Let $d \geq 1$.  Assume that there exists a function $L : {\mathbb S}^{d-1} \to \re$, bounded and measurable if $d \geq 2$, such that
    \bel{apr1}
    \lim_{r \to \infty} r^2 V(r\omega) = L(\omega),
    \ee
    uniformly with respect to $\omega \in {\mathbb S}^{d-1}$. If $d =1$, set
    $$
    \lambda_1(L) : = \min{\{L(-1), L(1)\}}, \quad \lambda_2(L) : = \max{\{L(-1), L(1)\}}.
    $$
    If $d \geq 2$, let $\left\{\lambda_j\right\}_{j \in {\mathbb N}}$ be the non decreasing sequence of the  eigenvalues of $-\Delta_{{\mathbb S}^{d-1}} + L$, where $-\Delta_{{\mathbb S}^{d-1}}$ is the Beltrami-Laplace operator, self-adjoint in $L^2({\mathbb S}^{d-1})$. Note that $\lambda_1(L)$ is a simple eigenvalue, and $\lambda_j(L) \to \infty$ as $j \to \infty$. For $d \geq 1$ set
    \bel{j10}
{\mathcal C}_d (L) : =
\frac{1}{2\pi} \sum_{j} \left(\lambda_j(L)+\frac{(d-2)^2}{4}\right)_-^{1/2}.
    \ee
\begin{proposition} \label{mp3}
 {\rm (\cite{KS1, hm})} Let $d \geq 1$. Assume that $V \in L^{\infty}(\re^d)$, and there exists a function $L : {\mathbb S}^{d-1} \to \re$, bounded and measurable if $d \geq 2$, such that \eqref{apr1} holds true.
Then we have
	\bel{j17}
\lim_{E \uparrow 0} (|\ln{|E||})^{-1} N(E;V) = {\mathcal C}_d (L).
	\ee
If, moreover, $\lambda_1(L) > - \frac{(d-2)^2}{4}$, $d \geq 1$, then \eqref{m11} is valid again.
\end{proposition}
{\em Remark}: If $L$ is a constant and $d \geq 2$, then $\lambda_j(L) = \lambda_j(0) + L$, $j \in {\mathbb N}$, where $\lambda_j(0)$  are the well-known eigenvalues of the Beltrami-Laplace operator (see, e.g. \cite[Subsections 22.3-4]{shu}); in particular, $\lambda_1(0) = 0$.\\

\section{Main Results}
\label{s2} \setcounter{equation}{0}
Let us first introduce several definitions needed for the statements of the main results.\\
We will write $W \in {\mathcal S}_{m,\varrho}(\re^d)$, $m \in {\mathbb Z}_+$, $\rho \in (0,\infty)$, if $W \in C^m(\re^d;\re)$, and there exists a constant $C \in (0,\infty)$ such that
\bel{m0}
|D^\alpha W(x)| \leq C (1+|x|)^{-\rho - |\alpha|}, \quad x \in \re^d,
\ee
for each $\alpha \in {\mathbb Z}_+^d$ with $0 \leq |\alpha| \leq m$. \\
Further, we define a class of admissible functions $\eta : \re^d \to \re$ which is quite similar to the one introduced in \cite[Subsection 2.1]{R3} (see also \cite[Subsection 2.2]{R2}).
  Assume that  $\eta = \eta_0
+ \tilde{\eta}$ where $\eta_0 \in \re$ is a constant, while the function
$0 \neq \tilde{\eta} : \re^d \to \re$ is such that the Poisson equation
     \bel{may1}
\Delta \varphi = \tilde{\eta}
    \ee
admits a  solution $\varphi \in C_{\rm b}^2(\re^d; \re)$, i.e. a solution
     $\varphi : \re^d \to \re$, continuous and bounded together with its derivatives of order up
to two. Note that if there exists a bounded solution  $\varphi \in C^2(\re^d)$ of \eqref{may1}, then due to the Liouville theorem (see e.g. \cite[Section 2.2, Theorem 8]{E}), this solution is unique up to an additive constant. Although this is not crucial for our analysis, we fix this arbitrary constant by choosing $\varphi$ so that
    \bel{may1a}
\limsup_{R \to \infty} \frac{\int_{\left(-\frac{R}{2},\frac{R}{2}\right)^d} \varphi(x) dx}{R^d} = 0.
    \ee
    Then we will say that $\eta$ is an {\em admissible} function, and will write $\eta \in {\rm Adm}(\re^d)$.
 Since
$$
\eta_0 = \lim_{R \to \infty} \frac{\int_{\left(-\frac{R}{2},\frac{R}{2}\right)^d} \eta(x) dx}{R^d},
$$
 we will call the
constant $\eta_0$ {\em the mean value} of $\eta$.
  Also, we will call $\tilde{\eta}$ {\em
the background} of $\eta$.
Next, we describe our leading example of admissible backgrounds $\tilde{\eta}$.
Let $\xi_n \subset \re^d\setminus\{0\}$, $n \in \N$, $\xi_{n_1} \neq \xi_{n_2}$ if $n_1 \neq n_2$, and $\eta_n \in {\mathbb
C}$, $n \in {\mathbb N}$. Assume that
    \bel{may1b}
0 < \sum_{n \in {\mathbb N}}
|\eta_n| (1 + |\xi_n|^{-2})< \infty.
    \ee
 Then the almost periodic
function
    \bel{may9}
\tilde{\eta}(x) : = \sum_{n \in {\mathbb N}} \eta_n
e^{i\xi_n . x}, \quad x \in \re^d,
    \ee
is an admissible background,
provided that it is real-valued. In this case
    \bel{may10}
    \varphi(x) : = -\sum_{n \in {\mathbb N}} \eta_n |\xi_n|^{-2}
e^{i\xi_n . x}, \quad x \in \re^d,
    \ee
    is the solution $\varphi \in C_{\rm b}^2(\re^d;\re)$ of \eqref{may1}, which satisfies also \eqref{may1a}.\\
    If $\eta = \eta_0 + \tilde{\eta}$ with $\eta_0 \in \re$, and real-valued  $\tilde{\eta}$ of form \eqref{may9} with $\eta_n$ and $\xi_n$ satisfying  \eqref{may1b}, then we will write $\eta \in {\mathcal A}(\re^d)$.\\
    For example, if $\eta : \re^d \to \re$ is a non identically constant function, periodic with respect to a non degenerate lattice in $\re^d$, with absolutely convergent series of Fourier coefficients, then $\eta \in {\mathcal A}(\re^d)$. \\
    {\em Remark}: Evidently, the class ${\rm Adm}(\re^d)$ is essentially larger than ${\mathcal A}(\re^d)$. For example, if $0 \neq \tilde{\eta} \in C_0^\infty(\re^d; \re)$ with $d \geq 3$, then $\tilde{\eta} \in {\rm Adm}(\re^d)$ since
    $$
    \varphi(x) : = - \frac{1}{d(d-2)\tau_d} \int_{\re^d} |x-y|^{2-d} \tilde{\eta}(y)\, dy, \quad x \in \re^d,
    $$
    provides a solution of \eqref{may1} which is in $C_{\rm b}^2(\re^d;\re)$ (see e.g. \cite[Section 2.2, Theorem 1]{E}), and satisfies \eqref{may1a} as it decays at infinity. On the other hand, obviously, $\tilde{\eta} \not \in {\mathcal A}(\re^d)$.
\begin{theorem} \label{th4}
Let $d \geq 1$, $\rho \in (0,2]$,  and $W \in {\mathcal S}_{2,\rho}(\re^d)$. Suppose that $\eta = \eta_0 + \tilde{\eta} \in {\rm Adm}(\re^d)$. \\
{\rm (i)} Let $\rho \in (0,2)$. Assume that  there exist constants $C \in (0,\infty)$, and $R \in (0,\infty)$, such that
    \bel{j41}
W(x) \leq - C|x|^{-\rho}, \quad x \in \re^d, \quad |x| \geq R.
    \ee
If $\eta_0 > 0$, then
	\bel{j14}
N(E;\eta W) =
N_{\rm cl}(E; \eta_0 W) \, (1 + o(1)) \asymp
|E|^{d(\frac{1}{2} - \frac{1}{\rho})}, \quad E \uparrow 0,
	\ee
the function $N_{\rm cl}$ being defined in \eqref{may67}.
If, on the contrary, $\eta_0 < 0$, then
\bel{1}
N(0;\eta W) < \infty.
\ee
{\rm (ii)}  Assume $\eta_0 = 0$.
Then we have
    \bel{j21}
N(E;\eta W) = \left\{
\begin{array} {l}
O\left(|E|^{\frac{d}{2}(1 - \frac{1}{\rho})}\right) \quad {\rm if} \quad \rho \in (0,1),\\
O\left(|\ln{|E|}|\right) \quad {\rm if} \quad \rho = 1,\\
O\left(1\right) \quad {\rm if} \quad \rho \in (1,2],\\
\end{array}
    \right.
 \quad E \uparrow 0.
    \ee
{\rm  (iii)} Let $\rho = 2$. Suppose that there exists a function $L : {\mathbb S}^{d-1} \to \re$, bounded and measurable if $d \geq 2$, such that
$$
\lim_{r \to\infty} r^2 W(r \omega) = L(\omega),
$$
uniformly with respect to $\omega \in {\mathbb S}^{d-1}$.
Then we have
	\bel{j19}
\lim_{E \uparrow 0} (|\ln{|E||})^{-1} N(E;\eta W) = {\mathcal C}_d (\eta_0 L),
	\ee
${\mathcal C}_d$ being defined in \eqref{j10}.
If, moreover, $\lambda_1(\eta_0  L) > - \frac{(d-2)^2}{4}$, then \eqref{1} holds true.
\end{theorem}
{\em Remark}: If $\eta \in L^{\infty}(\re^d; \re)$, $d \geq 1$, and  $W \in {\mathcal S}_{0,\rho}(\re^d)$ with $\rho \in (2,\infty)$, then Proposition \ref{mp1}  implies that \eqref{1} holds true again. \\

Our next theorem makes the result of Theorem \ref{th4} (ii) more precise under more restrictive assumptions on $\eta$. For its formulation we need an additional notation. Assume $\eta \in {\rm Adm}(\re^d)$, and set
    \bel{may11}
    \psi(x) = |\nabla \varphi(x)|^2, \quad x \in \re^d,
    \ee
    where $\varphi \in C^2_{\rm b}(\re^d;\re)$ is the solution of \eqref{may1} - \eqref{may1a}. Note that if $\psi \in {\rm Adm}(\re^d)$, then its mean value $\psi_0$ is non negative.
\begin{theorem} \label{th5}
Let $d \geq 1$,  Suppose that $\eta \in {\rm Adm}(\re^d)$, and $\eta_0 = 0$. Let $\psi \in {\rm Adm}(\re^d)$, and $\psi_0 > 0$.\\
{\rm (i)} Assume that $\rho \in (0,1)$,  $W \in {\mathcal S}_{2,\rho}(\re^d; \re)$,
and
  there exist constants $C \in (0,\infty)$, and $R \in (0,\infty)$, such that
    \bel{may1c}
W(x)^2 \geq C|x|^{-2\rho}, \quad x \in \re^d, \quad |x| \geq R.
    \ee
Then we have
	\bel{may38}
N(E;\eta W) = N_{\rm cl}(E;-\psi_0 W^2)
 \, (1 + o(1)) \asymp
|E|^{\frac{d}{2}(1 - \frac{1}{\rho})}, \quad E \uparrow 0.
	\ee
{\rm  (ii)} Let $\rho = 1$. Assume that $\frac{\partial \varphi}{\partial x_j} \in {\rm Adm}(\re^d)$, $j=1,\ldots,d$. Suppose that $W \in {\mathcal S}_{3,\rho}(\re^d; \re)$, and there exists a function ${\mathcal L} : {\mathbb S}^{d-1} \to \re$, bounded and measurable if $d \geq 2$, such that
$$
\lim_{r \to\infty} r^2 W^2(r \omega) = {\mathcal L}(\omega),
$$
uniformly with respect to $\omega \in {\mathbb S}^{d-1}$.
Then we have
	\bel{may39}
\lim_{E \uparrow 0} (|\ln{|E||})^{-1} N(E;\eta W) = {\mathcal C}_d (-\psi_0 {\mathcal L}).
	\ee
If, moreover, $\lambda_1(-\psi_0  {\mathcal L}) > - \frac{(d-2)^2}{4}$, then \eqref{1} holds true again.
\end{theorem}
{\em Remarks}: (i) If $\eta \in {\mathcal A}(\re^d)$, and $\psi \in {\mathcal A}(\re^d)$, then
$$
\psi_0 = \sum_{n \in \N} \frac{|\eta_n|^2}{|\xi_n|^2} \in (0,\infty).
$$
However, $\eta \in {\mathcal A}(\re^d)$ does not imply automatically $\psi \in {\mathcal A}(\re^d)$. In fact, we have
 $$
 \psi(x) = \sum_{m \in \N} \sum_{n \in \N} \frac{\eta_m \overline{\eta}_n \xi_m \cdot \xi_n}{|\xi_m|^2 |\xi_n|^2} e^{i(\xi_m - \xi_n)\cdot x} =
 \psi_0 + \sum_{\N \ni m \neq n \in \N} \frac{\eta_m \overline{\eta}_n \xi_m \cdot \xi_n}{|\xi_m|^2 |\xi_n|^2} e^{i(\xi_m - \xi_n)\cdot x},
 $$ and the set $\left\{\xi_m - \xi_n\right\}_{\N \ni m \neq n \in \N}$ could contain subsequences which converge arbitrarily fast to zero.\\
 On the other hand, if $\eta : \re^d \to \re$ is a non identically constant function, periodic with respect to a non degenerate lattice in $\re^d$, with absolutely convergent series of Fourier coefficients, then
$\psi \in {\mathcal A}(\re^d)$.\\
(ii) Note that if $\varphi \in C_{\rm b}^1(\re^d)$, then the mean values of the derivatives $\frac{\partial \varphi}{\partial x_j}$, $j=1,\ldots,d$, vanish. \\
(iii) It is easy to check that under the general assumptions of Theorem \ref{th5}, estimate \eqref{may65} holds if and only if $\rho > \frac{d}{d+1}$ which is coherent with the results of \cite{LNS, S1} and \cite[Theorem 1.2]{S3}. \\

The proof of Theorems \ref{th4} and \ref{th5} can be found in Section \ref{s4}.
 Note that most of the results of these theorems are not of semiclassical nature. For example, estimate \eqref{1} could hold true even if $N_{\rm cl}(0;\eta W) = \infty$.   This could happen if, for example, $W(x) = - (1+|x|^2)^{-\rho/2}$, $x \in \re^d$, $\rho \in (0,2)$, and  $\eta \in C^\infty({\mathbb T}^d; \re)$ with ${\mathbb T}^d = \re^d/{\mathbb Z}^d$, such that the mean value of $\eta$ is negative, but its positive part does not vanish identically. Similarly, relation \eqref{j14} is not semiclassical  in the sense of \eqref{j15} since it is not difficult to construct examples of $W$ and $\eta$ with $\eta_0 > 0$ which satisfy the assumptions of part (i) of Theorem \ref{th4}, such that
 $$
 \limsup_{E \uparrow 0}\frac{N_{\rm cl}(E;\eta_0 W)}{N_{\rm cl}(E;\eta W)} < 1.
 $$
 A related effect where the main terms of the eigenvalue asymptotics for quantum magnetic Hamiltonians depend only on the mean value of the oscillating {\em magnetic field}, can be found in \cite{R2, R3}. \\
 The non semiclassical nature of Theorem \ref{th5} is even more conspicuous, since, for instance, already the order $|E|^{\frac{d}{2}\left(1 - \frac{1}{\rho}\right)}$ of asymptotic relation  \eqref{may38} is different from the semiclassical one $|E|^{d\left(\frac{1}{2} - \frac{1}{\rho}\right)}$, appearing in \eqref{j15}.

\section{Proofs of the Main Results}
\label{s4} \setcounter{equation}{0}
{\bf 4.1.} In this subsection we introduce notations and establish auxiliary facts, necessary for the proofs of Theorems \ref{th4} and \ref{th5}. \\
Let $q$ be a a lower-bounded closed quadratic form with domain ${\rm Dom}(q)$, dense in the separable Hilbert space ${\mathcal H}$. Let $Q$ be the self-adjoint operator generated by $q$ in ${\mathcal H}$. Assume that $\sigma_{\rm ess}(Q) \subset [0,\infty)$. Let
    $\nu(q)$ denote the number of the negative eigenvalues of the operator $Q$, counted with the multiplicities.  By the well known Glazman lemma, we have
    $$
    \nu(q) = \sup \, {\rm dim}\,{\mathcal F},
    $$
    where ${\mathcal F}$ are the subspaces of ${\rm Dom}(q)$ whose non zero elements $u$ satisfy $q[u] < 0$.

If $q_j$, $j=1,2$,  are two lower-bounded closed quadratic forms, densely defined in the same Hilbert space ${\mathcal H}$, and ${\rm Dom}(q_1) \cap {\rm Dom}(q_2)$ is dense in ${\mathcal H}$, then
    \bel{may125}
    \nu(q_1+q_2) \leq \nu(q_1) + \nu(q_2)
    \ee
    (see e.g. \cite[Eq. (125)]{RS4}).  \\
    Further, let $\Omega \subset \re^d$, $d \geq 1$, be a domain, i.e. an open, connected, non empty set. If $d \geq 2$, and $\partial \Omega \neq \emptyset$, we assume that the boundary $\partial \Omega$ is Lipschitz, and will say that $\Omega$ is admissible. Let $E \in (-\infty,0]$, let $g : \Omega \to [c_1,c_2]$ with $0 < c_1 \leq c_2 < \infty$, be a Lebesgue-measurable function, and let $V : \Omega \to \re$ be a bounded, Lebesgue-measurable function, such that $\lim_{|x| \to \infty, \; x \in \Omega} V(x) = 0$ in case that $\Omega$ is unbounded. Introduce the quadratic forms
    $$
    q_{E,j}[u; \Omega, g, V] : = \int_\Omega \left(g\left(|\nabla u|^2 - E |u|^2\right) + V |u|^2\right)\, dx, \quad j = {\mathcal D}, {\mathcal N},
    $$
    with domains
    $$
    {\rm Dom}(q_{E,{\mathcal D}}) = {\rm H}_0^1(\Omega), \quad {\rm Dom}(q_{E,{\mathcal N}}) = {\rm H}^1(\Omega),
    $$
    where, as usual, ${\rm H}^1(\Omega)$ is the first-order Sobolev space with norm defined by
    $$
    \|u\|^2_{{\rm H}^1(\Omega)} = \int_{\Omega} \left(|\nabla u|^2 + |u|^2\right) dx,
    $$
    while ${\rm H}_0^1(\Omega)$ is the completion of $C_0^\infty(\Omega)$ in ${\rm H}^1(\Omega)$.
    If $\Omega = \re^d$, then ${\rm H}_0^1(\re^d) = {\rm H}^1(\re^d)$, and we write $q_{E}[u; \re^d, g, V]$ instead of $q_{E,j}[u; \re^d, g, V]$, $j = {\mathcal D}, {\mathcal N}$.
    If we want to indicate the dependence of $q_{E,j}$ only on the parameters $\Omega, g, V$ but not on its variable $u \in  {\rm Dom}(q_{E,j})$, we write
    $q_{E,j}(\Omega, g, V)$ instead of $q_{E,j}[\cdot; \Omega, g, V]$. \\
    If $\Omega$ is a bounded admissible domain, then the compactness of the embedding of ${\rm H}^1(\Omega)$ into $L^2(\Omega)$ easily implies
    \bel{may45}
    \nu(q_{E,j}(\Omega,g,V)) = O(1), \quad E \uparrow 0, \quad j = {\mathcal D}, {\mathcal N}.
    \ee
    Similarly, if ${\rm supp}\,V_-$ is contained in a bounded admissible domain $\tilde{\Omega} \subset \Omega$, then \eqref{may45} holds true for any admissible $\Omega$. \\
    Note that we have
    \bel{may30}
    N(E;V) = \nu(q_E(\re^d; 1, V)), \quad E \leq 0.
    \ee
    For $R \in (0,\infty)$ set
    $$
    B_R : = \left\{x \in \re^d \, | \, |x| < R\right\}, \quad {\mathcal O}_R : = \left\{x \in \re^d \, | \, |x| > R\right\}.
    $$
    \begin{lemma} \label{l31}
    Let $R \in (0,\infty)$. Then for each $\varepsilon \in (0,1)$ we have
    \bel{may19}
    \nu\left(q_{E,{\mathcal N}}(\obr, 1, V)\right) \leq \nu\left(q_{E,{\mathcal D}}(\obr, 1, (1-\varepsilon)^{-1}V)\right) + O(1), \quad E \uparrow 0.
    \ee
    \end{lemma}
    \begin{proof}
    Let $\zeta \in C^\infty({\mathcal O}_R; [0,1])$ such that $\zeta(x) = 0$ if $x \in {\mathcal O}_R \cap B_{3R/2}$, and $\zeta(x) = 1$ if $x \in {\mathcal O}_{2R}$. By \eqref{may125}, we have
    $$
    \nu(q_{E, {\mathcal N}}({\mathcal O}_R, 1, V)) \leq
    $$
    \bel{may15}
    \nu(q_{E, {\mathcal N}}({\mathcal O}_R, 1-\varepsilon, \zeta^2 V - (1-\eps)\zeta \Delta \zeta)) +
    \nu(q_{E, {\mathcal N}}({\mathcal O}_R, \varepsilon, (1-\zeta^2) V + (1-\eps)\zeta \Delta \zeta)).
    \ee
    Since ${\rm supp}\,((1-\zeta^2) V + (1-\eps)\zeta \Delta \zeta) \subset \overline{B}_{2R} \setminus B_R$ is compact, we have
    \bel{may17}
    \nu(q_{E, {\mathcal N}}({\mathcal O}_R, \varepsilon, (1-\zeta^2) V + (1-\eps)\zeta \Delta \zeta)) = O(1), \quad \quad E \uparrow 0.
    \ee
    Next,
    $$
    q_{E, {\mathcal N}}[u;{\mathcal O}_R, 1-\varepsilon, \zeta^2 V - (1-\eps)\zeta \Delta \zeta] =
    $$
    $$
    \int_{\obr} \left((1-\eps)\left(|\nabla u|^2 - E|u|^2 - \zeta \Delta \zeta |u|^2\right) + \zeta^2 V |u|^2\right) dx \geq
    $$
    $$
    \int_{\obr} \left((1-\eps)\left(\zeta^2\left(|\nabla u|^2 - E|u|^2\right) - \zeta \Delta \zeta |u|^2\right) + \zeta^2 V |u|^2\right) dx =
    $$
    \bel{may16}
    q_{E, {\mathcal N}}[\zeta u;{\mathcal O}_R, 1-\varepsilon, V], \quad u \in {\rm H}^1(\obr).
    \ee
    Since $\zeta u \in {\rm H}^1_0(\obr)$ if $u \in {\rm H}^1(\obr)$, we find that \eqref{may16} implies
    $$
    \nu(q_{E, {\mathcal N}}({\mathcal O}_R, 1-\varepsilon, \zeta^2 V - (1-\eps)\zeta \Delta \zeta)) \leq
    $$
    \bel{may18}
    \nu(q_{E, {\mathcal D}}({\mathcal O}_R, 1-\varepsilon, V)) = \nu(q_{E, {\mathcal D}}({\mathcal O}_R, 1, (1-\varepsilon)^{-1} V)).
    \ee
    Now \eqref{may15}, \eqref{may17}, and \eqref{may18}, entail \eqref{may19}.
    \end{proof}
    {\em Remark}: Results of the type of Lemma \ref{l31} are well known (see e.g. \cite[Lemma 4.10]{BSo1}). We formulate Lemma \ref{l31} in a form which is both suitable and sufficient for our purposes.\\

    {\bf 4.2.} In this subsection we obtain the key estimate used in the proofs of Theorems \ref{th4} and \ref{th5} (see \eqref{may26} below). Assume that $\eta \in {\rm Adm}(\re^d)$,
    and set
    $$
    \Phi : = \varphi W,
    $$
    where $\varphi \in C_{\rm b}^2(\re^d; \re)$ is the solution of \eqref{may1} -- \eqref{may1a}. Note that if $W \in {\mathcal S}_{0,\rho}$ with $\rho > 0$, then $\Phi \in {\mathcal S}_{0,\rho}$; in particular,
    \bel{may20}
    \lim_{|x| \to \infty} \Phi(x) = 0.
    \ee
    Moreover, if $W \in {\mathcal S}_{2,\rho}$, then the functions $D^\alpha\Phi$ with $0 \leq |\alpha| \leq 2$ are bounded in $\re^d$, and decay at infinity. Further, we have $\Delta \Phi = \tilde{\eta} W + 2 \nabla \varphi \cdot \nabla W + \varphi \Delta W$, and, hence,
    $$
    \eta W = \eta_0 W + \Delta \Phi + \tilde{V},
    $$
    where
    \bel{may46}
    \tilde{V} : = - 2 \nabla \varphi \cdot \nabla W - \varphi \Delta W.
    \ee
    Therefore,
    $$
    q_E[u; \re^d, 1, \eta W] = \int_{\re^d} \left(|\nabla u|^2 - E |u|^2 + \eta W |u|^2\right) dx =
    $$
    $$
    q_E[u; \re^d, 1, \eta W] = \int_{\re^d} \left(|\nabla u|^2 - E |u|^2 + (\eta_0 W + \Delta \Phi + \tilde{V}) |u|^2\right) dx =
    $$
    \bel{may25a}
    \int_{\re^d} \left(|\nabla u - \nabla \Phi u|^2 - E |u|^2 + (\eta_0 W - |\nabla \Phi|^2 + \tilde{V}) |u|^2\right) dx, \quad u \in {\rm H}^1(\re^d).
    \ee
    If $W \in {\mathcal S}_{1,\rho}$ with $\rho > 0$, then the mapping $u \mapsto e^{\Phi} u$ is an isomorphism in ${\rm H}^1(\re^d)$, and \eqref{may25a} implies $$
    q_E[e^{\Phi}u; \re^d, 1, \eta W] = q_E[u; \re^d, e^{2\Phi}, e^{2\Phi}(\eta_0 W - |\nabla \Phi|^2 + \tilde{V})], \quad u \in {\rm H}^1(\re^d).
    $$
    Hence,
    \bel{may25}
    \nu(q_E(\re^d, 1, \eta W)) = \nu(q_E(\re^d, e^{2\Phi}, e^{2\Phi}(\eta_0 W - |\nabla \Phi|^2 + \tilde{V})), \quad E \leq 0.
    \ee
    Pick $\eps \in (0,1)$, and find $R \in (0,\infty)$ such that
    \bel{may21}
    1-\eps \leq e^{2\Phi(x)} \leq 1 + \eps, \quad |x| > R,
    \ee
    which is possible due to \eqref{may20}.\\
    Applying a standard Dirichlet-Neumann bracketing, and bearing in mind \eqref{may21}, we get
    $$
    \nu(q_{E, {\mathcal D}}(\obr , 1, (1+\eps)^{-1}e^{2\Phi}(\eta_0 W - |\nabla \Phi|^2 + \tilde{V}))) \leq
    $$
    $$
    \nu(q_E(\re^d, e^{2\Phi}, e^{2\Phi}(\eta_0 W - |\nabla \Phi|^2 + \tilde{V}))) \leq
    $$
    $$
    \nu(q_{E, {\mathcal N}}(\obr , 1, (1-\eps)^{-1}e^{2\Phi}(\eta_0 W - |\nabla \Phi|^2 + \tilde{V}))) +
    $$
    $$
    \nu(q_{E, {\mathcal N}}(B_R, e^{2\Phi}, e^{2\Phi}(\eta_0 W - |\nabla \Phi|^2 + \tilde{V}))) =
    $$
    \bel{may24}
    \nu(q_{E, {\mathcal N}}(\obr , 1, (1-\eps)^{-1}e^{2\Phi}(\eta_0 W - |\nabla \Phi|^2 + \tilde{V}))) + O(1), \quad E \uparrow 0.
    \ee
    Further, it follows from \eqref{may19} and \eqref{may45} that
    $$
    \nu(q_{E, {\mathcal D}}(\obr , 1, (1+\eps)^{-1}e^{2\Phi}(\eta_0 W - |\nabla \Phi|^2 + \tilde{V}))) + O(1) \geq
    $$
    $$
    \nu(q_{E, {\mathcal N}}(\obr , 1, (1+\eps)^{-2}e^{2\Phi}(\eta_0 W - |\nabla \Phi|^2 + \tilde{V}))) \geq
    $$
    $$
    \nu(q_{E}(\re^d , 1, (1+\eps)^{-2}e^{2\Phi}(\eta_0 W - |\nabla \Phi|^2 + \tilde{V}))) -
    $$
    $$
    \nu(q_{E, {\mathcal N}}(B_R , 1, (1+\eps)^{-2}e^{2\Phi}(\eta_0 W - |\nabla \Phi|^2 + \tilde{V}))) =
    $$
    \bel{may22}
    \nu(q_{E}(\re^d , 1, (1+\eps)^{-2}e^{2\Phi}(\eta_0 W - |\nabla \Phi|^2 + \tilde{V}))) + O(1), \quad E \uparrow 0,
    \ee
    and
    $$
    \nu(q_{E, {\mathcal N}}(\obr , 1, (1-\eps)^{-1}e^{2\Phi}(\eta_0 W - |\nabla \Phi|^2 + \tilde{V}))) \leq
    $$
    $$
    \nu(q_{E, {\mathcal D}}(\obr , 1, (1-\eps)^{-2}e^{2\Phi}(\eta_0 W - |\nabla \Phi|^2 + \tilde{V}))) + O(1) \leq
    $$
    \bel{may23}
    \nu(q_{E}(\re^d , 1, (1-\eps)^{-2}e^{2\Phi}(\eta_0 W - |\nabla \Phi|^2 + \tilde{V}))) + O(1), \quad E \uparrow 0.
    \ee
    Combining \eqref{may30} with \eqref{may24} -- \eqref{may23}, we find that for each $\eps \in (0,1)$ we have
    $$
    N(E; (1+\eps)^{-1}e^{2\Phi}(\eta_0 W - |\nabla \Phi|^2 + \tilde{V}))  + O(1) \leq
    $$
    $$
    N(E; \eta W) \leq
    $$
    \bel{may26}
    N(E; (1-\eps)^{-1}e^{2\Phi}(\eta_0 W - |\nabla \Phi|^2 + \tilde{V}))  + O(1), \quad E \uparrow 0.
    \ee

    {\bf 4.3.} In this subsection we prove Theorem \ref{th4}. Assume at first the hypotheses of its part (i). Let $\eta_0 > 0$. Making use of \eqref{may125} and  \eqref{may26}, we find that for each $\eps \in (0,1)$ we have
    $$
    N(E; (1+\eps)^{-2} \eta_0 W) - N(E; -\eps^{-1}(1+\eps)^{-1}((e^{2\Phi}-1)\eta_0 W + e^{2\Phi}(- |\nabla \Phi|^2 + \tilde{V})))  + O(1) \leq
    $$
    $$
    N(E; \eta W) \leq
    $$
    \bel{may31}
    N(E; (1-\eps)^{-2}\eta_0 W) + N(E; \eps^{-1}(1-\eps)^{-1}((e^{2\Phi}-1)\eta_0 W + e^{2\Phi}(- |\nabla \Phi|^2 + \tilde{V})))  + O(1)
    \ee
    as $E \uparrow 0$.\\
    By \eqref{j15}, we have
    \bel{may32}
    N(E;(1+\eps)\eta_0 W) =
    N_{\rm cl}(E; (1+\eps)\eta_0 W) \, (1 + o(1)) \asymp
    |E|^{d(\frac{1}{2} - \frac{1}{\rho})}, \quad E \uparrow 0,
    \ee
    provided that $\eta_0 > 0$ and $\eps \in (-1,\infty)$.
     Further, it is not difficult to show that
     \bel{may33}
     \lim_{\eps \to 0} \limsup_{E \uparrow 0} \left| \frac {N_{\rm cl}(E; (1+\eps)\eta_0 W)}{N_{\rm cl}(E; \eta_0 W)} - 1\right| = 0
    \ee
    (see e.g. \cite[Subsection 3.7]{R} for a similar argument). \\
    Finally, it is easy to see that there exists a constant $C \in (0,\infty)$ such that
    $$
    \left|((e^{2\Phi}-1)\eta_0 W + e^{2\Phi}(- |\nabla \Phi|^2 + \tilde{V}))\right| \leq C (1 + |x|)^{-\rho_*}, \quad x \in \re^d,
    $$
    where
    \bel{may35}
    \rho_* : = \min\{2\rho, \rho + 1\}.
    \ee
    Therefore, Proposition \ref{mp2} implies that
    \bel{may34}
    N(E; c((e^{2\Phi}-1)\eta_0 W + e^{2\Phi}(- |\nabla \Phi|^2 + \tilde{V}))) =
     o\left(|E|^{d(\frac{1}{2} - \frac{1}{\rho})}\right), \quad E \uparrow 0,
    \ee
    for any $c \in \re$. Putting together \eqref{may31} -- \eqref{may34}, we obtain \eqref{j14}.\\
    Assume now that $\eta_0 < 0$. Then the support of the negative part of $e^{2\Phi}\left(\eta_0 W - |\nabla \Phi|^2 + \tilde{V}\right)$ is compact. Hence, the upper bound of \eqref{may26}, implies \eqref{1}. \\
    Next, assume the hypotheses of Theorem \ref{th4} (ii); in particular, $\eta_0 = 0$. Then there exists a constant $C \in (0,\infty)$ such that
    \bel{may36}
    \left(- |\nabla \Phi|^2 + \tilde{V}\right)_- \leq C (1 + |x|)^{-\rho_*}, \quad x \in \re^d,
    \ee
    $\rho_*$ being defined in \eqref{may35}. Note that $\rho_* = 2\rho$ if $\rho \in (0,1]$, and $\rho_* = \rho + 1 > 2$ if $\rho >1$. Then the result of
    Theorem \ref{th4} (ii) follows from the upper bound in \eqref{may26}, \eqref{may36}, and Propositions \ref{mp1}, \ref{mp2}, and \ref{mp3}.\\
    Finally, assume the hypotheses of Theorem \ref{th4} (iii). Bearing in mind \eqref{may31} and \eqref{j17}, we find that for every $\eps \in (0,1)$ we have
    $$
    {\mathcal C}_d ((1-\eps)\eta_0 L) \leq
    $$
    $$
   \liminf_{E \uparrow 0} (|\ln{|E||})^{-1} N(E;\eta W) \leq  \limsup_{E \uparrow 0} (|\ln{|E||})^{-1} N(E;\eta W) \leq
   $$
     \bel{j18}
     {\mathcal C}_d ((1+\eps)\eta_0 L).
     \ee
    Due to the continuity of the function $\re \ni \eta_0 \mapsto {\mathcal C}_d(\eta_0 L)$, we conclude that \eqref{j19} follows form \eqref{j18}. Moreover, if $\lambda_1(\eta_0 L) > - \frac{(d-2)^2}{4}$, then there exists $\eps \in (0,1)$ such that $\lambda_1((1+\eps)\eta_0 L) > - \frac{(d-2)^2}{4}$. Hence, the upper bound in \eqref{may31}, and Proposition \ref{mp3} imply again \eqref{1} in this case.\\

    {\bf 4.4.} In this subsection we prove Theorem \ref{th5}. Assume at first the hypotheses of its part (i); in particular, $\eta_0 = 0$ and $\rho \in (0,1)$. Then \eqref{may26} implies that for any $\eps \in (0,1)$ we have
    $$
    N(E; -(1+\eps)^{-2} \psi W^2) - N(E; -\eps^{-1}(1+\eps)^{-1}(e^{2\Phi}(- |\nabla \Phi|^2 + \tilde{V}) + \psi W^2))  + O(1) \leq
    $$
    $$
    N(E; \eta W) \leq
    $$
    \bel{may37}
    N(E; -(1-\eps)^{-2} \psi W^2) + N(E; \eps^{-1}(1-\eps)^{-1}(e^{2\Phi}(- |\nabla \Phi|^2 + \tilde{V}) + \psi W^2))  + O(1)
    \ee
    as $E \uparrow 0$. It is easy to check that there exists a constant $C \in (0,\infty)$ such that
    \bel{may48}
    \left|e^{2\Phi}(- |\nabla \Phi|^2 + \tilde{V}) + \psi W^2\right| \leq C (1 + |x|)^{-\rho_{**}}, \quad x \in \re^d,
    \ee
    where $\rho_{**} : = \min\{3\rho, \rho + 1\}$.
    Combining \eqref{may48} with Proposition \ref{mp2}, we get
    \bel{may39a}
    N(E; -\eps^{-1}(1+\eps)^{-1}(e^{2\Phi}(- |\nabla \Phi|^2 + \tilde{V}) + \psi W^2)) = o\left(|E|^{\frac{d}{2}\left(1-\frac{1}{\rho}\right)}\right),
    \quad E \uparrow 0.
    \ee
    Now \eqref{may38} easily follows from \eqref{may37}, \eqref{may39a}, Theorem \ref{th4} (i), and an analogue of \eqref{may33} with $\eta_0 W$ replaced by
    $-\psi_0 W^2$. \\
    Finally, assume  the hypotheses of Theorem \ref{th5} (ii); in particular $\eta_0 = 0$ and $\rho = 1$. Similarly to \eqref{may37}, we find that  \eqref{may26} implies
    $$
    N(E; -(1+\eps)^{-2} \psi W^2) -
    $$
    $$
    N(E; 2\eps^{-1}(1+\eps)^{-1}(e^{2\Phi} (\varphi \Delta W + |\nabla \Phi|^2) - \psi W^2 + 2(e^{2\Phi}-1)\nabla \varphi \cdot \nabla W))  -
    $$
    $$
    N(E; 4\eps^{-1}(1+\eps)^{-1}\nabla \varphi \cdot \nabla W)) + O(1) \leq
    $$
    $$
    N(E; \eta W) \leq
    $$
    $$
    N(E; -(1-\eps)^{-2} \psi W^2) +
    $$
    $$
    N(E; -2\eps^{-1}(1-\eps)^{-1}(e^{2\Phi} (\varphi \Delta W + |\nabla \Phi|^2) - \psi W^2 + 2(e^{2\Phi}-1)\nabla \varphi \cdot \nabla W))  +
    $$
    \bel{may40}
    + N(E; - 4\eps^{-1}(1-\eps)^{-1}\nabla \varphi \cdot \nabla W)) + O(1), \quad  E \uparrow 0,
    \ee
    for any $\eps \in (0,1)$ (see \eqref{may46}). It is easy to see that there exists a constant $C \in (0,\infty)$ such that
    $$
    \left|e^{2\Phi} (\varphi \Delta W + |\nabla \Phi|^2) - \psi W^2 + 2(e^{2\Phi}-1)\nabla \varphi \cdot \nabla W\right| \leq C (1 + |x|)^{-3}, \quad x \in \re^d,
    $$
    which combined with Proposition \ref{mp1}, implies
     \bel{may41}
     N(E; c (e^{2\Phi} (\varphi \Delta W + |\nabla \Phi|^2) - \psi W^2 + 2(e^{2\Phi}-1)\nabla \varphi \cdot \nabla W)) = O(1), \quad E \uparrow 0,
     \ee
     for any $c \in \re$. Further since $\frac{\partial \varphi}{\partial x_j}$, $j =1,\ldots,d$, are admissible functions with vanishing mean values, while
     $\frac{\partial W}{\partial x_j} \in {\mathcal S}_{2,2}(\re^d)$, $j =1,\ldots,d$, it follows from Theorem \ref{th4} (ii) that
     \bel{may42}
     N(E; c \nabla \varphi \cdot \nabla W)) =  O(1), \quad  E \uparrow 0,
     \ee
     for any $c \in \re$. Thus we find that, similarly to \eqref{j18}, estimates \eqref{may40} and \eqref{j17} imply
     $$
    {\mathcal C}_d (-(1-\eps)\psi_0 {\mathcal L}) \leq
    $$
    $$
   \liminf_{E \uparrow 0} (|\ln{|E||})^{-1} N(E;\eta W) \leq  \limsup_{E \uparrow 0} (|\ln{|E||})^{-1} N(E;\eta W) \leq
   $$
     \bel{may43}
     {\mathcal C}_d (-(1+\eps)\psi_0 {\mathcal L}).
     \ee
     Due to the continuity of the function $\re \ni \psi_0 \mapsto {\mathcal C}_d(-\psi_0 {\mathcal L})$, we conclude that \eqref{may39} follows form \eqref{may43}. Moreover, if $\lambda_1(-\psi_0 {\mathcal L}) > - \frac{(d-2)^2}{4}$, then there exists $\eps \in (0,1)$ such that $\lambda_1(-(1+\eps)\psi_0 {\mathcal L}) > - \frac{(d-2)^2}{4}$. Hence, the upper bound in \eqref{may40}, estimates \eqref{may41} -- \eqref{may42}, and Proposition \ref{mp3}, imply  \eqref{1}.\\

    {\bf Acknowledgements}. The first version of this work has been done during author's visits to the University
of Hagen, Germany, in December 2014 - January 2015, and to the Isaac Newton Institute,
Cambridge, UK, in January - February 2015.  The author thanks these institutions
for financial support and hospitality.  The partial support of the Chilean Scientific
Foundation Fondecyt
under Grant 1130591, and of
N\'ucleo Milenio de F\'isica Matem\'atica
RC120002, is gratefully acknowledged as well.\\
The author thanks the anonymous referee of an earlier version of this paper for a valuable suggestion, and Dr Marcello Seri for drawing his attention to the
articles \cite{skr, sas}.\\

\bigskip

{\sc G. Raikov}\\
Facultad de Matem\'aticas\\
Pontificia Universidad Cat\'olica de Chile\\
Av. Vicu\~na Mackenna 4860\\ Santiago de Chile\\
E-mail: graikov@mat.puc.cl\\

\end{document}